\documentclass[11pt, reqno]{amsart}

\usepackage{geometry}
\usepackage{graphicx}
\usepackage{amssymb}
\usepackage{amsmath, mathrsfs, stmaryrd}
\usepackage{tabulary}
\usepackage{booktabs}

\usepackage{color}
\newtheorem{theorem}{Theorem}
\newtheorem*{theorem*}{Theorem}

\newtheorem{lemma}[theorem]{Lemma}

\newcommand{\R}{\mathbb R}
\title{A rigidity theorem for surfaces in Schwarzschild manifold}
\author{Po-Ning Chen}
\address{Department of Mathematics\\ University of California, Riverside, CA}
\email{poningc@math.ucr.edu}
\author{Xiangwen Zhang}
\address{Department of Mathematics\\ University of California, Irvine, CA}
\email{xiangwen@math.uci.edu}

\begin{document}
\maketitle
\vspace{-0.6cm}

\begin{abstract}
In this article, we prove a rigidity theorem for isometric embeddings into the Schwarzschild manifold, by using the variational formula of quasi-local mass.
\end{abstract}

\

The classical Weyl problem asks whether any smooth metric on a two dimensional sphere with positive Gauss curvature admits a smooth isometric embedding into the three dimensional Euclidean space $\mathbb R^3$. The problem was solved by Nirenberg in his landmark paper \cite{Nirenberg}. A closely related important question is about the rigidity (or uniqueness) of such global isometric embeddings. Recall that a surface in $\mathbb R^3$ is called {\it rigid} if any other isometric surface is the same up to a rigid motion of $\R^3$. In 1927, Cohn-Vossen proved the following well-known rigidity theorem.
\begin{theorem*}[Cohn-Vossen \cite{Cohn-Vossen}]
Let $\Sigma$ and $\Sigma'$ be two closed convex surfaces in $\mathbb R^3$. Suppose $\Sigma'$ and $\Sigma$ are isometric. Then, $\Sigma$ and $\Sigma'$ are the same up to a motion or a motion and a reflection.
\end{theorem*}

The Weyl problem has been studied intensively with various generalizations, because it plays a fundamental role in convex and differential geometry and also has important applications in the definition of quasi-local mass in general relativity, see for example, \cite{Chang-Xiao, Chen-Wang-Wang-Yau, Guan-Li, Lu-Guan, Hong-Zuily, Lewy, Li-Wang, Lin-Wang, Liu-Yau1, Liu-Yau2, Lu, Shi-Wang-Yu, Wang-Yau1, Wang-Yau2}. Meanwhile, the rigidity theorem of Cohn-Vossen was also extended to more general settings \cite{Chern, Dajczer, Guan-Shen, Hsiung-Liu, Pogorelov, Pogorelov1, Sacksteder}. 

The purpose of the present paper is to extend Cohn-Vossen's theorem to closed convex surfaces in a Schwarzschild manifold. Recall that Schwarzschild manifold is a rotationally symmetric space $(M, g)$ with 
\begin{equation*}
 g = \bar V^{-2} d r^2 + r^2 \, dS^2, 
\end{equation*}
where $\bar V= \sqrt{1-\frac{2m}{r}}$ is the static potential with $r>2m$ and $dS^2$ is the standard metric on the sphere. The boundary at $r=2m$ is a minimal surface called the horizon.
%
\medskip

\begin{theorem}\label{maintheorem}
Let $\Sigma$ and $\Sigma'$ be closed strictly convex surfaces enclosing the horizon in $(M, g)$. 
Suppose $\Sigma'$ and $\Sigma$ are isometric and have the same mean curvature. Further, we assume that $Ric(\nu,\nu) \le 0$ on $\Sigma$ where $\nu$ is the normal vector of $\Sigma$.
Then $\Sigma$ and $\Sigma' $ are the same up to an isometry of the Schwarzschild manifold.
%
%
\end{theorem}
\medskip
In comparison with the Cohn-Vossen theorem in Euclidean space or its generalization in space forms \cite{Dajczer, Guan-Shen, Sacksteder}, an extra condition on the mean curvature of the isometric surfaces is imposed. Indeed, this condition is necessary due to the small isometry group of Schwarzschild space. It was observed in \cite{Li-Wang} that, in the Schwarzschild manifold, the $r$ radius sphere is not rigid. Namely, there exists some smooth perturbation convex body isometric to the $r$ sphere but their second fundamental forms are different. 
%

\

On the other hand, there are two natural motivations to impose the condition on mean curvature. The first motivation is the Alexandrov's uniqueness theorem \cite{Aleksandrov, Alex57, Chern, HNY, GWZ} which states that a closed strictly convex twice differentiable surface is uniquely determined to within a parallel translation by the given mean curvature. More precisely, consider two strictly convex surfaces $\Sigma$ and $\Sigma'$ in $\mathbb R^3$, with mean curvatures $H$ and $H'$ respectively. Using the Gauss map, we can view both $H$ and $H'$ as functions defined on the sphere. The uniqueness theorem says {\it if $H=H'$, then $\Sigma$ and $\Sigma'$ are the same up to a linear transformation.} Theorem \ref{maintheorem} can be viewed as an extension of Alexandrov's theorem from the point that, instead of using Gauss maps, we identify two surfaces in Schwarzschild space by {\it isometric} map.

Our second motivation to study such a Cohn-Vossen type rigidity theorem arising from the study of the weighted quasi-local Penrose inequality in \cite{Lu-Miao}. The quasi-local Penrose inequality gives a lower bound for the quasi-local mass with reference in the Schwarzschild manifold in terms of the area of the enclosed minimal surface. In particular, let $\Sigma$ be a convex surface with $Ric(\nu,\nu)\leq 0$ and encloses the horizon in a Schwarzschild manifold $M$, and $\Sigma'$ be an isometric and mean convex surface enclosing the horizon in $M$. The result in \cite{Lu-Miao} implies that $\int_{\Sigma} V  (H - H') \ge 0$
where $V$ is the restriction of the static potential $\bar V$ to $\Sigma$; $H$ and $H'$ are the mean curvatures of $\Sigma$ and $\Sigma'$, respectively. Moreover, if the equality holds (further discussions on the equality case can also be found in \cite{Shi-Wang-Yu}), then
\[ H= H'. \]
Therefore, it is a natural problem to further investigate the rigidity of the isometric surfaces with the same mean curvature. 

\

Unlike the pure PDE method used to prove the Alexandrov's uniqueness theorem (see for example \cite{GWZ, HNY}), we obtain the rigidity result by exploiting the positivity of the quasi-local mass. We will first compute a formula for the variation of the quasi-local mass with reference in a Schwarzschild manifold and then combine the formula with the quasi-local Penrose inequality. It might be worth to point out that the variational formula of quasi-local mass given in Lemma \ref{varmass} may have other applications.
\smallskip

Throughout the paper, we assume that $m>0$ since $m=0$ is simply the flat Euclidean space. Denote the covariant derivative of the Schwarzschild manifold by $D$ and the Ricci curvature of the Schwarzschild manifold by $R_{ij}$.  The static equation reads
\[
D_i D_j \bar V  = R_{ij}\bar V .
\]

Given a surface $\Sigma$ in $M$, 
we denote the covariant derivative and Laplace operator of the induced metric on $\Sigma$ by $\nabla$ and $\Delta$, respectively. Again, we use $V$ to denote the static potential $\bar V$ restricted to the surface $\Sigma$.
Furthermore, let $H$ and $h_{ab}$ be the mean curvature and second fundamental form of $\Sigma$. We have
\[
\Delta V  = -R_{\nu\nu}V  - H \nu(\bar V )
\]
where $\nu$ is the outward unit normal of $\Sigma$. 

\

First, we describe the general setup for the computation of the variation of the quasi-local mass. Suppose $\Sigma$ and $\Sigma'$ are isometric surfaces in $M$ with the same mean curvature. We identify functions and tensors on $\Sigma$ and $\Sigma'$ through the pull-back and push-forward of the surface isometry. 
Let $\Sigma'(s)$ be a smooth family of surfaces in $M$ such that 
\begin{equation}\label{ansatz}
 \frac{d}{ds} \Sigma'(s)|_{s=0} = F\nu', \ \ \ \ \Sigma'(0)=  \Sigma', 
\end{equation}
where $\nu'$ is the unit normal of $\Sigma'$ and $F$ is any smooth function defined on $\Sigma'$. For $s$ small, suppose $\Sigma(s)$ is a family of smooth surfaces in $M$ such that
\begin{equation} \label{ansatz1}
\Sigma(s) \text{ is  isometric to } \Sigma'(s), \ \ \ \Sigma(0)=  \Sigma.
\end{equation}

Then, we can define the quasi-local mass
\[
E(s) = \int_{\Sigma(s)} V (s) (H(s)-H'(s)) \, d\sigma(s),
\]
where $H(s)$ and $H'(s)$ are the mean curvature of $\Sigma(s)$ and $\Sigma'(s)$, respectively and $V (s)$ is the restriction of the static potential to $\Sigma(s)$ and $d\sigma(s)$ is the area element of $\Sigma(s)$. The following lemma gives the first variation of the quasi-local mass, which is the one of the key ingredients in the proof of rigidity.

\

\begin{lemma}\label{varmass}
Let $\Sigma(s)$ and $\Sigma'(s)$ be families of surfaces in $M$ satisfying \eqref{ansatz} and \eqref{ansatz1}.
Suppose $H(0) = H'(0) = H$.
Then, we have
\[
\frac{d}{ds}\bigg |_{s=0} E(s)= {1\over 2}\int_{\Sigma} F\, V(0) \, |h- h'| ^2\, d\sigma
\]
where $h_{ab}$ and $h'_{ab}$ be the second fundamental form of $\Sigma$ and $\Sigma'$, and $d\sigma$ denotes the area element on $\Sigma$.
\end{lemma}
\begin{proof}
We decompose the variation of  $\Sigma(s)$ into tangential part and normal part to $\Sigma$
\begin{eqnarray}\label{varsigma}
\frac{d}{ds} \Sigma(s)|_{s=0} = G\nu + P
\end{eqnarray}
where $P$ is tangent to $\Sigma$ and $G$ is a function on $\Sigma$.
Then, it follows from the linearization of the isometric embedding that 
\begin{eqnarray}\label{liee}
2 F h'_{ab} = 2Gh_{ab} + \nabla_a P_b + \nabla_b P_a.
\end{eqnarray}
Taking the trace and using the condition $H(0)=H'(0)=H$, we have
\begin{equation}
F = G + \frac{{\rm div} P}{H}
\end{equation}
where ${\rm div} P = \nabla^a p_a$ is the divergence of $P$ with respect to the induced metric on $\Sigma$. Replacing $G$ in the linearized isometric embedding equation (\ref{liee}),
\begin{equation} \label{iso_traceless_eq}
2F (h'-h)_{ab} = \nabla_a P_b + \nabla_b P_a - 2 \frac{h_{ab}}{H}\,  {\rm div} P.
\end{equation}

\medskip

On the other hand, with the assumption $H(0)= H'(0) = H$, we can compute
\[
\frac{d}{ds} E(s)|_{s=0}  = \int_{\Sigma} V (0)\,  \frac{d}{ds} \bigg|_{s=0}(H(s)-H'(s)) \, d\sigma
\]
Given the perturbations of $\Sigma'$ and $\Sigma$ in (\ref{ansatz}) and (\ref{varsigma}), the second variational formula can be computed as
\[
\begin{split}
\frac{d}{ds}  \bigg|_{s=0}H(s)=&  - \Delta G - (R_{\nu\nu} + |h|^2) G + P\cdot \nabla H\\
\frac{d}{ds}  \bigg|_{s=0}H'(s)= &  - \Delta F - (R_{\nu'\nu'} + |h'|^2) F.
\end{split}
\]
It follows that 
\[
\begin{split}
    &  \int_{\Sigma} V (0) \, \frac{d}{ds} \bigg|_{s=0} (H(s)-H'(s))\, d\sigma\\
 =& \int_{\Sigma}  V(0)\,  \left [ - \Delta G - (R_{\nu\nu} + |h|^2) G + P\cdot \nabla H +\Delta F + (R_{\nu'\nu'} + |h'|^2) F  \right ] \, d\sigma\\
 = & \int_{\Sigma} V(0)\,  \left [ \Delta \left(\frac{{\rm div} P}{H}\right) + \left(R_{\nu'\nu'} + |h'|^2 -R_{\nu\nu} - |h|^2\right)F  + P\cdot\nabla H  + (R_{\nu\nu} + |h|^2)\frac{{\rm div} P}{H}     \right ]\, d\sigma.
\end{split}
\]
To simplify the notation, we will denote $V (0)$ by $V$ in the rest computation.
Using the Gauss equation, we have
\[
\begin{split}
K = &- R_{\nu\nu} + \frac{1}{2} (H^2 - |h|^2 ) \\
K'=   & -R _{\nu' \nu'} +  \frac{1}{2} (H^2 - |h'|^2 ).
\end{split}
\]
where $K$ and $K'$ denote the Gauss curvature of $\Sigma$ and $\Sigma'$, respectively. Indeed, $K=K'$ since the surfaces are isometric. Subtracting the two equations, we get 
\[
\frac{1}{2} ( |h'|^2 - |h|^2) = R_{\nu\nu} - R_{\nu'\nu'} .
\]
Recall that 
\[
\Delta V  = -R_{\nu\nu}\, V  - H  \nu(\bar V ).
\]
Using this equation and integration by parts
\[
\int_{\Sigma} V\, \Delta \left( {{\rm div} P\over H }\right) \, d\sigma= \int_{\Sigma} \left(-R_{\nu\nu}V - H\, \nu(\bar V)\right) \,{{\rm div} P\over H }\, d\sigma.
\]
Putting these computations back to the above variation, we have
\[
\frac{d}{ds}\bigg|_{s=0} E(s)  =\int_{\Sigma}  \frac{V }{2} ( |h'|^2 - |h|^2)F  + V  P\cdot \nabla H  + \left(\frac{V  |h|^2}{H} - \nu(\bar V )\right) {\rm div} P \, d\sigma.
\]
Next, we compute
\[
\nabla_a \nu(\bar V ) = D_{a} D_{\nu} \bar V  + h_{ab} \nabla^b V  = R_{a\nu}\, V+ h_{ab} \nabla^b V .
\]
As a result,
\begin{eqnarray}\label{varE}
&&\frac{d}{ds} \bigg|_{s=0}E(s)\\\nonumber
&=& \int_{\Sigma} \frac{V }{2} ( |h'|^2 - |h|^2)F  + V  P\cdot \nabla H  +  P^a (V R_{a\nu}+ h_{ab} \nabla^b V  )  +  \frac{V  |h|^2}{H} {\rm div} P \, d\sigma.
 \end{eqnarray}
On the other hand, we contract \eqref{iso_traceless_eq} with $-V  h_{ab}$ and integrate
\[
\begin{split}
     \int_{\Sigma} FV  h \cdot (h-h') \, d\sigma
 = &  \int_{\Sigma}   \frac{V  |h|^2}{H} \, {\rm div} P-  V  h_{ab}\nabla^a P^b  \, d\sigma \\
 = & \int_{\Sigma}   \frac{V  |h|^2}{H} \, {\rm div} P + P^b \nabla^a (V  h_{ab}) \, d\sigma \\
  = &  \int_{\Sigma}  \frac{V  |h|^2}{H}\, {\rm div} P  + P^b \nabla^aV  h_{ab}  + V  P^b \nabla^a h_{ab} \, d\sigma \\
 = &  \int_{\Sigma}  \frac{V  |h|^2}{H} \, {\rm div} P + P^b \nabla^aV  h_{ab}  + V  P^b \nabla_b H  + V  P^bR_{b\nu} \, d\sigma,
 \end{split}
\]
where the Codazzi equation is used in the last equality. It follows that 
\[
\int_{\Sigma} FV  h \cdot (h-h')=\int_{\Sigma}  V  P\cdot \nabla H  +  P^a (V Ric(a,\nu)+ h_{ab} \nabla^b V  )  +  \frac{V  |h|^2}{H}\, {\rm div} P \, d\sigma.
\]
Finally, using this equation to replace the last three terms on the right hand side of (\ref{varE}), we obtain
\begin{eqnarray*}
\frac{d}{ds} \bigg|_{s=0}E(s)
&=& \int_{\Sigma} \frac{V }{2} ( |h'|^2 - |h|^2)F + FV  h \cdot (h-h')\, d\sigma\\
&=& {1\over 2} \int_{\Sigma} F\, V\, |h-h'|^2\, d\sigma.
\end{eqnarray*} 
\end{proof}

\medskip 


Using the variational formula, we can now prove Theorem \ref{maintheorem}. In fact, we can prove a generalized version which does not require the convexity of the surface $\Sigma'$.

\begin{theorem}
Let $\Sigma$ be a convex surface enclosing the horizon in the Schwarzschild manifold $(M, g)$, and $\Sigma'$ is also a surface enclosing the horizon in $M$ and is isometric to $\Sigma$. Suppose $\Sigma'$ and $\Sigma$ have the same mean curvature and $Ric(\nu,\nu) \le 0$ on $\Sigma$. Then $\Sigma$ and $\Sigma' $ are the same up to an isometry of the Schwarzschild manifold.
\end{theorem}

\begin{proof}
We first show that 
\[
h_{ab} = h'_{ab}.
\]
Consider a smooth family of surfaces $\Sigma'(s)$ in $M$ such that 
\[
\frac{d}{ds}\bigg|_{s=0} \Sigma'(s) = \nu', \ \ \ \Sigma'(0)=  \Sigma'.
\]
From the openness of isometric embedding into warp product space \cite{Li-Wang}, for $s$ sufficiently small, there exist a smooth family of surfaces $\Sigma(s)$ isometric to $\Sigma'(s)$ such that 
\[
\Sigma(0) = \Sigma
\]
since $\Sigma$ is strictly convex. We consider the quasi-local mass
\[
E(s) = \int_{\Sigma(s)} V (s) (H(s)-H'(s)) \, d\sigma(s).
\]
For $s$ sufficiently small, $\Sigma(s)$ is convex and $\Sigma'(s)$ is mean convex. As a result, 
\[
E(s) \ge 0
\]
for $s$ sufficiently small by the quasi-local Penrose inequality \cite{Lu-Miao}. On the other hand, the assumption $H(0)=H'(0)=H$ implies
\[
E(0)=0.
\]
We conclude that 
\[
\frac{d}{ds} \bigg|_{s=0}E(s)= 0.
\]
From Lemma 2, taking $F\equiv 1$, it follows that 
\[
\int_{\Sigma}  V  |h- h'| ^2  \, d\sigma = 0
\]
and thus 
\[
h = h'.
\]
Using the Gauss and the Codazzi equations, it is easy to see that the norm of the Ricci curvature tensor are the same for isometric surfaces. 
The Schwarzschild manifold is rotationally symmetric and the norm of the Ricci curvature tensor is a strictly decreasing function of  the radial function $r$.
This implies that the restrictions of $r$ to the two surfaces are the same, under the identification of the surface isometry. From here, one can use the rotation in the Schwarzschild manifold to identify the two surfaces.\\
\end{proof}

As the present work neared completion, we were informed that Chunhe Li, Pengzi Miao and Zhizhang Wang have a work in progress on related results.



\bigskip

\noindent {\bf Acknowledgement:} The authors would like to thank Pengzi Miao for the nice talk given in the UCI-UCR-UCSD joint seminar, which brought the problem to their attention. We would also like to thank Siyuan Lu for his interest and comments. The second author is supported by Simons Foundation collaboration grant.

\


\begin{thebibliography}{99} 
\bibitem{Aleksandrov} A. D. Alexandrov, \textit{Uniqueness theorems for surfaces in the large. I}, Vestnik Leningrad. Univ. {\bf 11} (1956), no.19, 5-17.
\smallskip

\bibitem{Alex57}A. D. Alexandrov, {\em Uniqueness theorems for surfaces in the large. II}, Vestnik Leningrad. Univ. 12 (1957) no. 7, 15-44.
\smallskip

\bibitem{Chang-Xiao}J.-E. Chang and L. Xiao, {\em The Weyl problem with nonnegative Gauss curvature in hyperbolic space}, Canad. J. Math. 67 (2015), no. 1, 107-131.
\smallskip

\bibitem{Chen-Wang-Wang-Yau} P.-N. Chen, M.-T. Wang, Y.-K. Wang and S.-T. Yau, \textit{Quasi-local energy with respect to a static spacetime}, arXiv:1604.02983.
\smallskip

\bibitem{Chern} S. S. Chern, {\em Integral formulas for hypersurfaces in Euclidean space and their applications to uniqueness theorems}, J. Math. Mech, 8, (1959), 947-956.
\smallskip

\bibitem{Cohn-Vossen} S. Cohn-Vossen,  \textit{Zwei Sätze über die Starrheit der Eiflachen}, Nachr. Ges. Wiss. zu Göttingen, (1927), 125-134.
\smallskip

\bibitem{Dajczer} M. Dajczer, {\em Submanifolds and isometric immersions}. Based on the notes prepared by Mauricio Antonucci, Gilvan Oliveira, Paulo Lima-Filho and Rui Tojeiro. Mathematics Lecture Series,13. Publish or Perish, Inc., Houston, TX, 1990.
\smallskip

\bibitem{Guan-Li} P. Guan and Y. Li, {\em On Weyl problem with nonnegative Gauss curvature}, J. Differential Geom., 39 (1994), 331-342.
\smallskip

\bibitem{Lu-Guan} P. Guan and S. Lu, {\em Curvature estimates for immersed hypersurfaces in Riemannian manifolds}, Invent. Math. 208 (2017), no. 1, 191-215.
\smallskip

\bibitem{Guan-Shen} P. Guan and X. Shen, {\em A rigidity theorem for hypersurfaces in higher dimensional space forms}, 
Contemporary Mathematics, AMS. V.644, 2015, 61-65.
\smallskip

\bibitem{GWZ} P. Guan, Z. Wang and X. Zhang, {\em A proof of the Alexandrov's uniqueness theorem for convex surfaces in $\mathbb R^3$}, Ann. Inst. H. Poincar\'e Anal. Non-Lin\'eaire, 33 (2016), 329-336.
\smallskip

\bibitem{HNY} Q. Han, N. Nadirashvili and Y. Yuan, {\em Linearity of homogeneous order-one solutions to elliptic equations in dimension three}, Comm. Pure and App. Math., Vol LVI(2003), 0425-0423.
\smallskip

\bibitem{Hong-Zuily} J. Hong and C. Zuily, {\em Isometric embedding of the 2-sphere with nonnegative curvature in $\mathbb R^3$}, Math. Z., 219 (1995), 323-334.
\smallskip

\bibitem{Hsiung-Liu} C. Hsiung and J. Liu, {\em A generalization of the rigidity theorem of Cohn-Vossen}, J. London Math. Soc. (2) 15 (1977), no. 3, 557-565.
\smallskip

\bibitem{Lewy} H. Lewy, {\em On the existence of a closed convex surface realizing a given Riemannian metric}, Proceedings of the National Academy of Sciences, U.S.A., Volume 24, No. 2, (1938), 104-106.
\smallskip

\bibitem{Li-Wang} C. Li\ and\ Z. Wang, \textit{The Weyl problem in warped product space}, arXiv:1603.01350. J. Differential Geom., to appear.
\smallskip

\bibitem{Lin-Wang} C.-Y. Lin and Y.-K. Wang, {\em On Isometric embeddings into Anti-de Sitter Space-times}, Int. Math. Res. Not., (2015), no. 16, 7130-7161.
\smallskip

\bibitem{Liu-Yau1} C.-C. M. Liu and S.-T. Yau, {\em Positivity of quasilocal mass}, Phys. Rev. Lett. 90 (2003) No. 23, 231102.
\smallskip

\bibitem{Liu-Yau2} C.-C. M. Liu and S.-T. Yau, {\em Positivity of quasilocal mass II}, J. Amer. Math. Soc. 19 (2006) No.1, 181-204.
\smallskip

\bibitem{Lu} S. Lu, {\em  On Weyl's embedding problem in Riemannian manifolds}, arXiv:1608.07539.
\smallskip

\bibitem{Lu-Miao} S. Lu and P. Miao, \textit{Minimal hypersurfaces and boundary behavior of compact manifolds with nonnegative scalar curvature}, arXiv:1703.08164.
\smallskip

\bibitem{Nirenberg} L. Nirenberg, {\em The Weyl and Minkowski problems in differential geometry in the large}, Comm. Pure Appl. Math. 6, (1953), 337-394.
\smallskip

\bibitem{Pogorelov} A. V. Pogorelov, {\em Some results on surface theory in the large}, Adv. in Math. 1(1964), no.2, 191-264.
\smallskip

\bibitem{Pogorelov1} A. V. Pogorelov, {\em Extrinsic geometry of convex surfaces}, Translated from the Russian by Israel Program for Scientific Translations. Translations of Mathematical Monographs, Vol. 35. American Mathematical Society, Providence, R.I., 1973.
\smallskip

\bibitem{Sacksteder} R. Sacksteder, {\em The rigidity of hypersurfaces}, J. of Math. Mech. 11 (1962), 929-939.
\smallskip

\bibitem{Shi-Wang-Yu} Y. Shi, W. Wang and H. Yu, {\rm On the Rigidity of Riemannian-Penrose Inequality for Asymptotically Flat 3-manifolds with Corners}, arXiv:1708.06373
\smallskip

\bibitem{Wang-Yau1} M.-T. Wang and S.-T. Yau, {\em A generalization of Liu-Yau's quasi-local mass}, Comm. Anal. Geom., V.15, (2007), 249-282.
\smallskip

\bibitem{Wang-Yau2}  M.-T. Wang and S.-T. Yau, {\em Isometric embeddings into the Minkowski space and new quasi-local mass}, Comm. Math. Phys. 288 (2009), no. 3, 919-942.

\end{thebibliography}
\end{document}